\documentclass[12pt]{amsart}
\usepackage{amssymb,amscd, amsmath, amsfonts, enumerate}

\usepackage{graphicx}
\usepackage{color}

\DeclareMathAlphabet{\mathssbx}{OT1}{cmss}{bx}{n}

\begin{document}

\setcounter{tocdepth}{1}

\def\iff{if, and only if,\ }
\def\g0{G^{(0)}}
\def\h0{H^{(0)}}
\def\maD{\mathcal{D}}
\def\maJ{\mathcal{J}}
\def\maK{\mathcal{K}}
\def\codim{\mathrm{codim}}

\renewcommand{\theenumi}{\alph{enumi}}
\renewcommand{\labelenumi}{\rm {({\theenumi})}}
\renewcommand{\labelenumii}{(\roman{enumii})}
%
%

\newcommand\ind{\operatorname{ind}}
\newcommand\End{\operatorname{End}}
\newcommand\per{\operatorname{per}}
\newcommand\pa{\partial}
\newcommand\sign{\operatorname{sign}}
\newcommand\supp{\operatorname{supp}}

\newcommand\CI{\mathcal{C}^\infty}
\newcommand\CO{\mathcal{C}_0}

\newcommand\CC{\mathbb C}
\newcommand\NN{\mathbb N}
\newcommand\RR{\mathbb R}
\newcommand\QQ{\mathbb Q}
\newcommand\ZZ{\mathbb Z}

\newcommand\CIc{{\mathcal C}^{\infty}_{\text{c}}}
\newcommand{\GR}{\mathcal G}


\newcommand\tgt[1]{{}^{T}\kern-1pt #1}
\newcommand\adi[1]{{}^{ad}\kern-1pt #1}
\newcommand{\fa}{\mathfrak{A}}
\newcommand{\faa}{\mathfrak{a}}

\newcommand\ie{{\em i.e.,\ }}
%
%
\newtheorem{theorem}{Theorem}
\newtheorem{proposition}{Proposition}
\newtheorem{corollary}{Corollary}
\newtheorem{lemma}{Lemma}
\newtheorem{definition}{Definition}
\newtheorem{notation}{Notations}
\theoremstyle{remark}
\newtheorem{remark}[theorem]{Remark}
\newtheorem{example}[theorem]{Example}
\newtheorem{examples}[theorem]{Examples}

%

%
%
\def\n#1#2{\| #1 \|_{{#2}}}
\def\g0{\GR^{(0)}}
\def\S{\mathscr{S}}
\def\cred#1{C^*_{\mathrm{r}}(#1)}
\def\rep#1{{\R_+^*}\!^{#1}}  
\def\r+#1{\R_+^{#1}}
\def\cinfo{{\mathcal C}^{\infty,0}}
\def\ccinfo{{\mathcal C}_c^{\infty,0}}
\def\Cc{{\mathcal C}_c}
\let \mx \mbox
\let \hx \hbox
\let \vx \vbox

\def\C{\mathbb{C}}
\def\E{\mathbb{E}}
\def\A{\mathbb{A}}
\def\B{\mathbb{B}}
\def\D{\mathbb{D}}
\def\F{\mathbb{F}}
\def\H{\mathbb{H}}
\def\N{\mathbb{N}}
\def\P{\mathbb{P}}
\def\Q{\mathbb{Q}}
\def\R{\mathbb{R}}
\def\Z{\mathbb{Z}}
\def\G{\mathbb{G}}
\def\K{\mathbb{K}}
\def\J{\mathbb{J}}
\def\T{\mathbb{T}}

\def\diagr#1{\def\normalbaselines{\baselineskip=0pt
\lineskip=10pt\lineskiplimit=1pt} \matrix{#1}}
\def\hfl#1#2{\smash{\mathop{\hx to 12mm{\rightarrowfill}}
\limits^{\scriptstyle#1}_{\scriptstyle#2}}}
\def\vfl#1#2{\llap{$\scriptstyle #1$}\left\downarrow
\vx to 6mm{}\right.\rlap{$\scriptstyle #2$}}
\def\antihfl#1#2{\smash{\mathop{\hx to 12mm{\leftarrowfill}}
\limits^{\scriptstyle#1}_{\scriptstyle#2}}}
\def\antivfl#1#2{\llap{$\scriptstyle #1$}\left\uparrow
\vx to 6mm{}\right.\rlap{$\scriptstyle #2$}}

\def\build#1_#2^#3{\mathrel{\mathop{\kern 0pt#1}\limits_{#2}^{#3}}}

\def\limind{\mathop{\oalign{lim\cr\hidewidth$\longrightarrow$\hidewidth\cr}}}
\def\limproj{\mathop{\oalign{lim\cr \hidewidth$\longleftarrow$\hidewidth\cr}}}

\def\norme#1{\left\| #1 \right\|}
\def\module#1{\left| #1 \right|}
\def\va#1{\left| #1 \right|}
\def\scal#1{\left\langle #1 \right\rangle}
\def\inv#1{{#1}^{-1}}
\def\invf{f^{-1}}

\def\psd{pseudodiff\'erentiel}
\def\dg{\partial G}
\def\dm{\partial M}
\def\dx{\partial X}

\def\intm{\overset{\:\bullet}{M}}
\def\intx{\overset{\:\bullet}{X}}
\def\intf#1{\overset{\:\bullet}{#1}}
\def\cc#1{C^*(#1)}
\def\kc#1{K_*(C^*(#1))}
\def\rep#1{{\R_+^*}^{#1}}
\def\r+#1{\R_+^{#1}}

\def\tg{groupo\"\i de tangent}
\def\gr{groupoid}
\def\d{{\rm d}}
\def\e{{\varepsilon}}
\def\cstar{$C^*$-algebra}
\def\rbar{\overline\R}
\def\l{\lambda}
\def\kth{$K$-theory}
\def\cinf{$C^\infty$}
\def\M{M}
\def\G{{\bf G}}
\def\E{\mathcal{E}}
\def\In{\hbox{In}}
\def\pd{pseudodifferential}

\def\k{\mathssbx{k}}

\def\lb{\hbox{lb}}
\def\rb{\hbox{rb}}

\def\bsp{$b$-stretched product}
\def\bcalc{$b$-calculus}

\def\sm{submanifold}
\def\mc{manifold with corners}
\def\mecs{manifolds with embedded faces}
\def\mec{manifold with embedded faces}
\def\mcs{manifolds with corners}

\def\D{\hbox{D}}
\def\S{\mathcal S}

\def\A{\mathcal A}
\def\F{\mathcal F}
\def\HF{\mathcal{HF}}
\def\sym{\mathfrak S}
\def\GG{\mathcal G}
\def\sh{\mathrm{sh}}
\def\ch{\mathrm{ch}}

\def\fr{\frac}
\def\ub{\underbar}
\def\O{\Cal O}
\def\F{\Cal F}
\def\differ{\text{differentiable} }
\def\tPSeudo{\text{pseudodifferential} }
\def\supp{\text{supp} }
\def\inn{{\mathcal R}}
\def\frag{\frak{G}}
\def\simd{\tilde{d}}
\def\simf{\tilde{\F}}
\def\simo{\tilde{\O}}
\def\simr{\tilde{r}}
\def\simp{\tilde{p}}
\def\simmu{\tilde{\mu}}
\def\O12{\Omega^\frac{1}{2}}

\def\maG{\mathcal{G}}
\def\maH{\mathcal{H}}
\def\maK{\mathcal{K}}
\def\maI{\mathcal{I}}
\def\maL{\mathcal{L}}
\def\maV{\mathcal{V}}
\newcommand{\Diff}{\operatorname{Diff}}
\newcommand{\Diffb}{\Diff_b(M)}
\newcommand{\DiffVM}{\operatorname{\Diff_{\maV}(M)}}
\newcommand{\PsiVM}{\Psi_\maV^\infty(M)}
\newcommand{\Psib}{\Psi_b^\infty(M)}

\date\today
\author[B. Monthubert]{Bertrand Monthubert}
       \address{ Universit\'e Paul
       Sabatier (UFR MIG), Institut de Math\'ematiques de Toulouse,
       F-31062 Toulouse CEDEX 4}
       \email{bertrand.monthubert@math.univ-toulouse.fr}

\author[V. Nistor]{Victor Nistor} \address{Pennsylvania State
       University, Math. Dept., University Park, PA 16802}
       \email{nistor@math.psu.edu}

\thanks{Monthubert was partially supported by a ACI Jeunes
       Chercheurs. Manuscripts available from {\bf
       http:{\scriptsize//}bertrand.monthubert.net}.
       Nistor was partially supported by the NSF Grant DMS
       0555831. Manuscripts available from {\bf
http:{\scriptsize//}www.math.psu.edu{\scriptsize/}nistor{\scriptsize/}}.}

\begin{abstract}
We compute the $K$-theory of comparison $C^*$-algebra associated to a
manifold with corners. These comparison algebras are an example of the
abstract pseudodifferential algebras introduced by Connes and
Moscovici \cite{M3}. Our calculation is obtained by showing that
the comparison algebras are a homomorphic image of a groupoid
$C^*$-algebra. We then prove an index theorem with values in the
$K$-theory groups of the comparison algebra.
\end{abstract}

\title[Comparison $C^*$-algebras]{The $K$-groups and the index theory
  of certain comparison $C^*$-algebras}

\maketitle \tableofcontents

\section*{Introduction}

The work of Henri Moscovici encompasses many areas of mathematics,
most notably Non-commutative Geometry, Group Representations,
Geometry, and Abstract Analysis. His work on Non-commutative geometry,
mostly joint works with Alain Connes, has lead to many breakthroughs
in Index Theory and Operator Algebras, as well as to applications to
other areas. We are happy to dedicate this paper to Henri Moscovici on
the occasion of his 65th birthday.

The problem studied in this paper pertains to the general program of
understanding index theory on singular and non-compact spaces. On such
spaces, the Fredholm property depends on more than the principal
symbol, so cyclic cocycles are needed in order to obtain explicit
index formulas. Moscovici has obtained many results in this direction,
including \cite{M7, M4, M6, M3, M5,M8}.  See also
\cite{BenameurHeitsch, Gorokhovsky, Connes1, Connes2, NestTsygan}.

One of the central concepts in a recent paper by Connes and Moscovici,
is that of an abstract algebra of pseudodifferential operators
\cite{M3}. These algebras generalize similar algebras introduced
earlier. In this paper, we would like to study certain natural
$C^*$-algebras associated to non-compact Riemannian manifolds,
applying in particular the point of view of the work of Connes and
Moscovici mentioned above.

Let us now explain the framework of this paper. Let $M_0$ be a
complete Riemannian manifold and let $\Delta = d^*d$ be the {\em
  positive} Laplace operator on $M_0$ associated to the metric. It is
well known that $\Delta$ is essentially self-adjoint \cite{Cordes92,
  Strichartz} and the references therein, and hence we can define
$\Lambda = (1 + \Delta)^{-1/2}$ using functional calculus. Let us also
assume that a certain algebra $\maD = \cup \maD_n$ of differential
operators is given on $M_0$, where $\maD_n$ denotes the space of
differential operators in $\maD$ of degree at most $n$. Let us assume
that $\Delta \in \maD_2$ and that $L_n \Lambda^n$ defines a bounded
operator on $L^2(M_0)$ for any $L_n \in \maD_n$.  Then the {\em
  comparison algebra} of $M_0$ (and $\maD$) to be the $C^*$-algebra
generated by the operators of the form $L_n \Lambda^n$ for any $L_n
\in \maD_n$. This definition is almost the same as the on in
\cite{Cordes77, CordesBook}, where the comparison algebra was defined
as the $C^*$-algebra generated by all operators of the form $L_1
\Lambda$ for any $L_1 \in \maD_1$ {\em and all compact operators}. One
of our results, Theorem \ref{theorem.4}, implies that the two
definitions are the same for suitable $M_0$. Let us denote the
comparison $C^*$-algebra of $M_0$ by $\fa(M_0)$ (the dependence on the
algebra $\maD$ will be implicit). The comparison algebra $\fa(M_0)$ is
a convenient tool to study many analytic properties of differential
operators on $M_0$, such as invertibility between Sobolev spaces,
spectrum, compactness, the Fredholm property, and the index
\cite{ConnesNCG, Cordes77, Georgescu, LMN1, LNgeometric,
  TaylorGelfand}. For instance, the principal symbol of order zero
pseudodifferential operators extends to a continuous map $\sigma_0 :
\fa(M_0) \to C(S^*A)$ with kernel denoted $\fa_{-1}(M_0)$, where
$S^*A$ is a suitable compactification of the cosphere bundle of
$T^*M_0$.

In this paper, we concentrate on the index properties of elliptic
operators on a certain class of non-compact manifolds, called
``manifolds with poly-cylindrical ends.''  Recall that a {\em manifold
  with poly-cylindrical ends} is, locally, a product of manifolds with
cylindrical ends.  Our index depends only on the principal symbol, so
it takes values in the $K$-theory of the $C^*$-algebra
$\fa_{-1}(M_0)$. Index calculations are sometimes necessary in
applications, for instance in the study of Hartree's equation
\cite{HunsickerNS} and in the study of boundary value problems on
polyhedral domain \cite{HenggMN}.

More precisely, let us first assume that the given manifold $M_0$ is
the interior of the space of units of a groupoid Lie $\GR$. Then the
Lie groupoid structure of $\GR$ gives rise to a natural algebra of
differential operators $\maD$ on $M_0$, such as in the case of
singular foliations \cite{ConnesNCG, Skandalis}. In general, there
will be no natural metric on $M_0$, and even if a metric is chosen on
$M_0$, the associated Laplace operator $\Delta \not\in \maD$. However,
if $M_0$ is a Lie manifold \cite{ALN1}, then a natural class of
metrics exists on $M_0$ and $\Delta \in \maD$ for any metric in this
class. Recall that $M_0$ is a {\em Lie manifold} if the tangent bundle
$TM_0$ extends to a bundle $A \to M $ on a compactification $M$ of
$M_0$ to a manifold with corners such that the space of smooth section
$\maV := \Gamma(A)$ of $A$ has a natural Lie algebra structure induced
by the Lie bracket of vector fields and such that that the
diffeomorphisms generated by vector fields in $\maV$ preserve the
faces of $M$ \cite{ALN1}.

We shall show that the comparison algebra of a Lie manifold $M_0$
identifies with a subalgebra of a homomorphic image of a groupoid
(pseudodifferential) algebra. For any manifold with corners, we shall
denote by $\maV_M$ the Lie algebra of all vector fields tangent to all
faces of $M$. Then $\maV_M = \Gamma(A_M)$ for a unique (up to
isomorphism) vector bundle $A_M \to M$. If the vector bundle $A \to M$
defining a Lie manifold $M$ satisfies $A = A_M$, then we shall say
that $M_0$, the interior of $M$, is a {\em manifold with
  poly-cylindrical ends}. In that case, we prove that $\fa(M_0)$ is
(isomorphic to) the norm closure of $\Psi^0(\maG)$.  We then use this
result to compute the $K$-theory of the algebra $\fa_{-1}(M_0)$ and
the index in $K_0(\fa_{-1}(M_0))$ of elliptic operators in the
comparison algebra of a manifold with poly-cylindrical ends.

Let us explain in a little more detail our results. Let us assume that
our algebra $\maD$ of differential operators is generated by $\CI(M)$
and $\maV$. Let $D \in \maD$, then the principal symbol of $D$ extends
to a symbol defined on $A^*$. Assume that $D$ is elliptic, in the
sense that its principal symbol is invertible on $A^*$ outside the
zero section. Then the $K$-theory six-term exact sequence applied to
the tangent (or adiabatic) groupoid of $\GR$ defines a map
\begin{equation}
    \ind_a = \ind_a^M : K^0(A^*) \to K_0(\fa_{-1}(M_0)).
\end{equation}
One of our main results is a computation of the groups
$K_0(\fa_{-1}(M_0))$ and of the map $\ind_a$ in case $M_0$ is a
manifold with poly-cylindrical ends.

Our calculation of the group $K_0(\fa_{-1}(M_0))$ is as follows.
Consider an embedding $\iota : M \to X$ of manifolds with corners and
let $\iota_{!}$ be the push-forward map in $K$-theory. Morita
equivalence then gives rise to a morphism $\iota_* : K_0(C^*(M)) \to
K_0(C^*(X))$. Then Theorem \ref{commutativity} states that the
following diagram commutes
\begin{equation}\label{diag.I}
\begin{CD}
    K_0(C^*(M)) @>{\iota_*}>> K_0(C^*(X))\\
    @A{\ind_a^M}AA  @AA{\ind_a^X}A\\
    K^0(A_M^*) @>{\iota_{!}}>>  K^0(A_X^*).\\
\end{CD}
\end{equation}

If the manifold with corners $X$ is such that the natural morphisms
$\iota_* : K_0(C^*(M)) \to K_0(C^*(X))$ and $\ind_a^X : K^0(A_X^*) \to
K_0(C^*(X))$ are isomorphisms, we are going to say that $X$ is a {\em
  classifying space} for $M$. In that case, we can interpret the above
diagram as a topological index theorem in the usual sense. We also
obtain an identification of the groups $K_0(C^*(M))$ and of the map
$\ind_a^M$.

Let us now very briefly summarize the contents of the paper. In
Section \ref{grComp.index}, we introduce comparison algebras and we
show that they are closely related to groupoid algebras. We show that
the groupoid $C^*$-algebra $\overline{\Psi^0(\GR_M)}$ and the
comparison $C^*$-algebra $\fa(M_0)$ are in fact isomorphic for $M_0$ a
manifold with poly-cylindrical ends with compactification $M$. In
Section \ref{a.index}, we recall the definition of the full
$C^*$-analytic index using the tangent groupoid. In the process, we
establish several technical results on tangent groupoids. Section
\ref{sec.prop} contains the main properties of the full $C^*$-analytic
index. In this section, we also introduce the morphism $j_*$
associated to an embedding of manifolds with corners $j$ and we
provide conditions for $j_*$ and $\ind_a^M$ to be isomorphisms.  We
also discuss he compatibility of the full $C^*$-analytic index and of
the shriek maps. This is then used to establish the equality of the
full $C^*$-analytic and principal symbol topological index. Some of
the proofs in this paper are only sketched. See \cite{MN} for full
details.

\thanks{We thank Bernd Ammann, Catarina Carvalho, Sergiu Moroianu, and
  Georges Skandalis for useful discussions.  The second named author
  would like to thank the Max Planck Institute for Mathematics Bonn,
  where part of this work was completed, for hospitality and support.}

\section{Groupoids and comparison algebras\label{grComp.index}}

We shall need to consider pseudodifferential operators on groupoids
\cite{bm-fp, NWX}.  For simplicity, {\em we shall assume from now on
  that all our manifolds with corners have embedded faces.}

\subsection{Pseudodifferential operators on groupoids}
Throughout this paper, we shall fix a Lie groupoid $\GR$ with units
$M$ and Lie algebroid $A = A(\GR)$.  Here $M$ is allowed to have
corners.  To $\GR$, there is associated the pseudodifferential
calculus $\Psi^{\infty}(\GR)$, whose operators of order $m$ form a
linear space denoted $\Psi^{m}(\GR)$, $m \in \RR$, see \cite{bm-fp,
  NWX}.  For short, this calculus is defined as follows. Let $s : \GR
\to M$ be the source map and $\GR_x = s^{-1}(x)$. Then
$\Psi^{m}(\GR)$, $m \in \Z$, consists of smooth families of classical,
order $m$ pseudodifferential operators $(P_x \in \Psi^m(\GR_x))$, $x
\in M$, that are right invariant with respect to multiplication by
elements of $\GR$ and are ``uniformly supported.''  To define what
uniformly supported means, let us observe that the right invariance of
the operators $P_x$ implies that their distribution kernels $K_{P_x}$
descend to a distribution $k_P \in I^m(\GR, M)$ \cite{bm-these,
  NWX}. Then the family $P = (P_x)$ is called {\em uniformly
  supported} if, by definition, $k_P$ has compact support in $\GR$.

We then have the following result \cite{LMN1, bm1, NWX}.

\begin{theorem}\label{thm.quant}\  The space
$\Psi^{\infty}(\GR)$ is a filtered algebra, closed under adjoints, so
  that the usual principal symbol of pseudodifferential operators
  defines a surjective
\begin{equation*}
    \sigma_\GR^{(m)} : \Psi^m(\GR) \to
    S^m_{cl}(A^*)/S^{m-1}_{cl}(A^*),
\end{equation*}
with kernel $\Psi^{m-1}(\GR)$, for any $m \in \ZZ$.
\end{theorem}

\subsection{Comparison algebras}
We shall denote by $\pi$ the natural action of $\Psi^{\infty}(\GR)$ on
$\CIc(M_0)$ or on its completions and by $\faa_{-\infty}$ the
completion of $\pi(\Psi^{-\infty}(\GR))$ acting on all Sobolev spaces
$H^{-m}(M_0) \to H^{m}(M_0)$.  Let us denote $\faa_n :=
\pi(\Psi^{n}(\GR)) + \faa_{-\infty}$. The following result was proved
in \cite{LMN2} (see also \cite{ALN1, LNgeometric}).

\begin{theorem}\label{thm.faa}\ 
The space $\faa := \pi(\Psi(\GR)) + \faa_{-\infty}$ is a filtered
algebra by the pseudodifferential degree such that $\faa_0 :=
\pi(\Psi^0(\GR)) + \faa_{-\infty}$ consists of bounded operators, is
closed under the adjoint, and is spectrally invariant. Moreover,
$\faa_{-\infty}$ is a two-sided ideal of $\faa$ and 
\begin{equation*} 
	\Lambda := (1 + \Delta)^{-1/2} \in \faa_{-1} 
	:= \pi(\Psi^{-1}(\GR)) + \faa_{-\infty}.
\end{equation*}
\end{theorem}

From the above theorem we obtain that the comparison algebra is a
subalgebra of the norm closure of $\faa_0$.

\begin{theorem} \label{theorem.3}
Let $M_0$ be a Lie manifold. Then we have that
\begin{equation*}
  \fa(M_0) \subset \overline{\faa_0} = \pi(\overline{\Psi^0(\GR)}).
\end{equation*}
Moreover, $\fa(M_0)$ contains all compact operators.
\end{theorem}

\begin{proof} 
Let $L_n \in \maD_n$. Since $\Lambda \in \faa_{-1}$ and $\maD_n
\subset \faa_n := \pi(\Psi^{-1}(\GR)) + \faa_{-\infty}$.  It follows
that $L_n \Lambda^n \in \faa_0$, and hence the result.

To show that $\fa(M_0)$ contains the subalgebra of compact operators,
let us notice first that $e^{-t\Delta} \in \fa(M_0)$, since it can be
written as a function of $\Lambda := (1 + \Delta)^{-1/2}$. Since
$\CIc(M_0) \subset \CI(M)$, we have that $\phi e^{-t\Delta} \psi
\subset \fa(M_0)$, and hence the later contains compact operators. To
show that all compact operators are in $\fa(M_0)$, it is enough to
show that $\fa(M_0)$ has no invariant subspaces. That is, it is enough
to show that if $f, g \in L_2(M_0)$ are such that the inner product
$(f ,\phi e^{-t\Delta} \psi g)$ is zero for all $\phi, \psi \in
\CIc(M_0)$, then either $f=0$ or $g=0$. Indeed, let us choose $\phi_n,
\psi_n \in \CIc(M_0)$ such that $\phi_n f \to |f|$ and $\psi_n g \to
|g|$ in $L^2(M_0)$. Then
\begin{equation*}
	(|f| ,e^{-t\Delta} |g|) = \lim_{n\to \infty} (f ,\phi_n
  e^{-t\Delta} \psi_n g) = 0,
\end{equation*} 
which implies that either $f=0$ or $g=0$, since the heat kernel
$e^{-t\Delta}$ has positive distribution kernel.
\end{proof}

Let $M_0$ be a manifold with poly-cylindrical ends and $\GR$ be a
groupoid such that $A(\GR) = A_M$, where, we recall, $A_M \to M$ is a
vector bundle such that $\maV_M := \Gamma(A_M)$ consists of all smooth
vector fields on $M$ that are tangent to all the faces of $M$.  A
groupoid $\GR$ with this property is said {\em to integrate} $A_M$,
and is not unique. However, if the fibers of the source map $s : \GR
\to M$ are all connected and simply-connected, then $\GR$ is unique
(up to isomorphism) and will be denoted $\GR_M$. For $\GR_M$ the
vector representation $\pi$ is injective \cite{LNgeometric, bm1}. We
shall also denote by $C^*(M) = C^*(\GR_M)$. Recall that $C^*(\GR_M) =
\overline{\Psi^{-1}(\GR)}$, \cite{LNgeometric, bm1}.

The algebra $\Psi(\GR_M)$ was considered before by many authors,
including \cite{MelroseScattering, ScSc1, SchulzeBook}. For this
algebra, we actually have equality in the above theorem.

\begin{theorem} \label{theorem.4} 
Let $M_0$ be a manifold with poly-cylindrical ends.  Then
\begin{equation*}
  \fa(M_0) = \overline{\faa_0} \simeq \overline{\Psi^0(\GR_M)}.
\end{equation*}
\end{theorem}

\begin{proof} 
Let us recall that for manifolds with poly-cylindrical ends the vector
representation $\pi$ is injective on the norm closure
$\overline{\Psi^0(\GR_M)}$. We shall thus identify $\Psi^0(\GR_M)$ with
$\pi(\Psi^0(\GR_M))$. Since the principal symbol map acting on both
$\fa(M_0)$ and on $\overline{\Psi^0(\GR_M)}$ has the same range, namely
$C(S^*A_M)$, it is enough to show that $\fa_{-1}(M_0) =
\overline{\Psi^{-1}(\GR)} = C^*(M)$.

Let us notice that we can consider families, so proving $\fa_{-1}(M_0)
= C^*(M)$ is equivalent to proving $\fa_{-1}(M_0) \otimes C_0(X)=
C^*(M)\otimes C_0(X)$.  Moreover, the inclusion $\fa_{-1}(M_0) \subset
C^*(M)$ is compatible with the natural representations of $C^*(M)$
associated to the faces of $M$, as seen from their construction in
\cite{MelroseNistor}. It is enough then to prove that we have
isomorphisms on subquotients defined by these representations, which
are all of the form $\fa_{-1}(X_0) \otimes C_0(X)$, for some lower
dimensional manifolds. The proof proof finally reduces to show that
$\maK \subset \fa_{-1}(M_0)$, where $\maK$ is the algebra of compact
operators. For this we use Theorem~\ref{theorem.3}
\end{proof}

For manifolds with cylindrical ends (that is when $M$ has no corners
of codimension two or higher), this theorem was proved before in
\cite{CordesDoong}.

\section{The analytic index\label{a.index}}

\subsection{The adiabatic and tangent groupoids}
For the definition and study of the full $C^*$-analytic index, we
shall need the adiabatic and tangent groupoids associated to a
differentiable groupoid $\GR$. We now recall their definition.

Let $\GR$ be a Lie groupoid with space of units $M$. We construct both
the \textit{adiabatic groupoid } $\adi{\GR}$ and the \textit{tangent
  groupoid} $\tgt{\GR}$ \cite{ConnesNCG, Landsman, LMN1, bm-fp,
  Ramazan}. The space of units of $\adi{\GR}$ is $M \times [0,\infty)$
  and the tangent groupoid $\tgt{\GR}$ will be defined as the
  restriction of $\adi{\GR}$ to $M \times [0,1]$. The underlying set
  of the groupoid $\adi{\GR}$ is the disjoint union:
\begin{equation*}
        \adi{\GR} = A(\GR) \times \{ 0 \}\, \cup\,
        \GR \times (0,\infty).
\end{equation*}
We endow $A(\GR) \times \{ 0 \}$ with the structure of commutative
bundle of Lie groups induced by its vector bundle structure. We endow
$\GR \times (0,\infty)$ with the product groupoid structure. Then the
groupoid operations of $\adi{\GR}$ are such that $A(\GR) \times \{ 0
\}$ and $\GR \times (0,\infty)$ are subgroupoids with the induced
structure. Now let us endow $\adi{\GR}$ with a differentiable
structure.  To do so, it is enough to specify $A(\adi{\GR})$, since
its knowledge completely determines the differentiable structure of
$\adi{\GR}$ \cite{nistorINT}.  Then
\begin{equation}\label{eq.smooth.gad}
    \Gamma(A(\adi{\GR})) = t \Gamma( A(\GR \times [0, \infty))).
\end{equation}
More precisely, consider the product groupoid $\GR \times [0, \infty)$
  with pointwise operations. Then a section $X \in \Gamma(A(\GR \times
  [0, \infty)))$ can be identified with a smooth function $[0, \infty)
      \ni t \to X(t) \in \Gamma(A(\GR))$. We then require
      $\Gamma(A(\adi{\GR})) = \{tX(t)\}$, with $X \in \Gamma(A(\GR
      \times [0, \infty)))$.

It is easy to show that

\begin{lemma}\label{cor.prod}\
Let $\maH = \maG \times \RR^n$, as above. We have that $C^*(\maH_{ad})
\simeq C^*(\maG_{ad}) \otimes \CO(\RR^n)$ and that $C^*(\tgt{\maH})
\simeq C^*(\tgt{\maG}) \otimes \CO(\RR^n)$, the tensor product being
the (complete, maximal) $C^*$--tensor product.
\end{lemma}

\subsection{The full $C^*$-analytic index}
For each $t\in [0,1]$, $M \times \{t\}$ is a closed invariant subset
of $M \times [0, \infty)$ for the adiabatic and tangent groupoids, and
  hence we obtain an {\em evaluation morphism}
\begin{equation*}
    e_t : \cc{\tgt{\GR}} \to \cc{\tgt{\GR}_{M \times \{t\}}},
\end{equation*}
and, in particular, an exact sequence
\begin{equation}\label{eq.exact1}
    0 \to \cc{\tgt{\GR}_{M \times (0,1]}} \to \cc{\tgt{\GR}}
     \xrightarrow{e_0} C^*(A(\GR)) \to 0.
\end{equation}
Since $K_*( \cc{\tgt{\GR}_{M \times (0,1]}})=K_*(\cc{\GR} \otimes
      \CO((0,1])=0$, the evaluation map $e_0$ is induces an
    isomorphism in \kth.

The $C^*$-algebra $C^*(A(\GR))$ is commutative and we have
$C^*(A(\GR)) \simeq \CO(A^*(\GR))$. Therefore
$K_*(C^*(A(\GR)))=K^*(A^*(\GR))$. In turn, this isomorphism allows us
to define the \textit{full $C^*$-analytic index} $\ind_a$ as the
composition map
\begin{equation}\label{def.an.index1}
    \ind_a^\GR = e_1 \circ e_0^{-1} : K^*(A^*(\GR))
    \to K_*(\cc{\GR}),
\end{equation}
where $e_1 : \cc{\tgt{\GR}} \to \cc{\tgt{\GR}_{M \times \{1\}}} =
\cc{\GR}$ is defined by the restriction map to $M \times \{1\}$.  The
definition of the full $C^*$-analytic index gives the following.

\begin{proposition}\label{prop.comp}\
Let $\GR$ be a Lie groupoid with Lie algebroid $\pi : A(\GR) \to
M$. Also, let $N \subset F \subset M$ be a closed, invariant subset
which is an embedded submanifold of a face $F$ of $M$. Then the full
$C^*$-analytic index defines a morphism of the six-term exact
sequences associated to the pair $(A^*(\maG), \pi^{-1}(N))$ and to the
ideal $\cc{\maG_{N^c}} \subset \cc{\maG}$, $N^c := M \smallsetminus N$
\begin{equation*}
  \begin{CD}
    K^0(\pi^{-1}(N^c))  @>>> K^0(A^*(\maG))
    @>>> K^0(\pi^{-1}(N)) @>>> K^1(\pi^{-1}(N^c))\\
    @VVV @VVV @VVV @VVV\\
    K^0(\cc{\maG_{N^c}})  @>>> K^0(\cc{\maG})
    @>>> K^0(\cc{\maG_{N}}) @>>> K^1(\cc{\maG_{N^c}})\\
  \end{CD}
\end{equation*}
\end{proposition}

\begin{proof}\
The six-term, periodic long exact sequence in \kth\ associated to the
pair $(A^*(\maG), \pi^{-1}(N))$ is naturally isomorphic to the
six-term exact sequence in \kth\ associated to the pair $\CO(A^*_{M
  \smallsetminus N}) \subset \CO(A^*(\maG))$. The result follows from
the naturality of the six-term exact sequence in \kth\ and the
definition of the full $C^*$-analytic index~\eqref{def.an.indexM}.
\end{proof}

Recall that for $M$ a smooth manifold with corners with embedded
faces, we have denoted $A(\GR_M)=A_M$ and $\cc M := C^*(\GR_M)$. Then
the full $C^*$-analytic index becomes the desired map
\begin{equation}\label{def.an.indexM}
    \ind_a^M : K^*(A^*_M) \to K_*(\cc{M}).
\end{equation}
See \cite{Debord2, Debord1} for more properties of the analytic index.

\begin{remark}\ Assume $M$ has no corners (or boundary).
Then $\GR_M = M \times M$ is the product groupoid and hence
$\Psi^\infty(\GR_M) = \Psi^\infty(M)$. In particular, $C^*(M) :=
C^*(\GR_M) \simeq \maK$, the algebra of compact operators on $M$.  In
this case $K_0(C^*(M))=\ZZ$, and $\ind_a$ is precisely the analytic
index as introduced by \cite{AS1}. This construction holds also for
the case when $M$ is not compact, if one uses pseudodifferential
operators of order zero that are ``multiplication at infinity,'' as in
\cite{carvalho}.
\end{remark}

\section{Properties of the full $C^*$-analytic
index\label{sec.prop}}

The following proposition is an important step in the proof of our
index theorem, Theorem \ref{thm.topological}.

\begin{proposition}\label{prop.2}\
Let $X$ be a \mec\ such that each open face of $X$ is diffeomorphic to
a Euclidean space. Then the full $C^*$-analytic index
\begin{equation*}
    \ind_a^X : K^*(A^*_X) \to K_*(\cc{X}),
\end{equation*} 
defined in Equation \eqref{def.an.indexM}, is an isomorphism.
\end{proposition}

\begin{proof}\
The proof is by induction on the number of faces of $X$ using
Proposition \ref{prop.comp}, the six-term exact sequence in
$K$-theory, and the Five Lemma in homological algebra.
\end{proof}

\begin{remark}
The above proposition can be regarded as a Baum--Connes isomorphism
for manifolds with corners.
\end{remark}

\begin{proposition}\label{prop.3}
Let $\iota: M \to X$ be a closed embedding of manifold with
corners. Assume that, for each open face $F$ of $X$, the intersection
$F \cap M$ is a non-empty open face of $M$ and that every open face of
$M$ is obtained in this way. Then $\kc{M} \to \kc{X}$ is an
isomorphism denoted $\iota_*$.
\end{proposition}

\begin{proof} Recall from \cite{renault-equiv}
that two locally compact groupoids $G$ and $H$ are {\em equivalent}
provided there exists a topological space $\Omega$ and two continuous,
surjective open maps $r : \Omega \to \g0$ and $d : \Omega \to \h0$
together with a left (respectively right) action of $G$ (respectively
$H$) on $\Omega$ with respect to $r$ (respectively $d$), such that $r$
(respectively $d$) is a principal fibration of structural \gr\ $H$
(respectively $G$). An important theorem of Muhly--Renault--Williams
states that if $G$ and $H$ are equivalent, then $\kc {G} \simeq \kc
{H}$ \cite{renault-equiv}.

Our result then follows from the fact that $\Omega := r^{-1}(M)$
establishes the desired equivalence between $\GR_M$ and $\GR_X$.
\end{proof}

We can now prove a part of our principal symbol topological index
theorem, Theorem \ref{thm.topological}, involving an embedding $\iota
: M \to X$ of our manifold with corners $M$ into another manifold with
corners $X$. This theorem amounts to the fact that the diagram
\eqref{diag.I} is commutative. In order to prove this, we shall first
consider a tubular neighborhood
\begin{equation}\label{eq.def.jk}
    M \overset{k}{\hookrightarrow} U \overset{j}{\hookrightarrow} X
\end{equation}
of $M$ in $X$, so that $\iota = j \circ k$. The diagram (\ref{diag.I})
is then decomposed into the two diagrams below, and hence the proof of
the commutativity of the diagram (\ref{diag.I}) reduces to the proof
of the commutativity of the two diagrams below, whose morphisms are
defined as follows:\ the morphism $k_*$ is defined by Proposition
\ref{prop.3} and the morphis $j_*$ is defined by the inclusion of
algebras. The morphism $\iota_*$ is defined by $\iota_* = j_* \circ
k_*$. Finally, the morphism $k_!$ is the push-forward morphism.

Let us now turn our attention to the following diagram:
\begin{equation}\label{diag2}
\begin{CD}
    \kc{M} @>{k_*}>> \kc{U} @>{j_*}>>\kc{X}\\
    @A{\ind_a^M}AA @A{\ind_a^U}AA @AA{\ind_a^X}A\\ K^*(A^*_M)
    @>{k_{!}}>> K^*(A^*_U) @>{j_!}>> K^*(A^*_X).\\
\end{CD}
\end{equation}

The commutativity of the left diagram is part of the following
proposition, which is the most technical part of the proof. Its proof
is obtained by integrating a Lie algebroid obtained as a double
deformation of a tangent space.

\begin{proposition}\label{prop.diag1}
Let $\pi : U \to M$ be a vector bundle over a manifold with corners
$M$ and let $k : M \to U$ be the ``zero section'' embedding. Then the
following diagram commutes:
\begin{equation}\label{diag3}
\begin{CD}
    \kc{M} @>{k_*}>{\simeq}> \kc{U}\\
    @A{\ind_a^M}AA @A{\ind_a^U}AA\\
    K^*(A^*_M) @>{\simeq}>{k_{!}}> K^*(A^*_U)
\end{CD}
\end{equation}
\end{proposition}

The commutativity of the second square in the Diagram \ref{diag2}
follows from the naturality of the tangent groupoid construction.

\begin{proposition}\label{prop.diag2}\
Let $j : U \to X$ be the inclusion of the open subset $U$. Then the
diagram below commutes:
\begin{equation*}
\begin{CD}
    \kc{U} @>{j_*}>>\kc{X}\\
    @A{\ind_a^U}AA @AA{\ind_a^X}A\\
    K^*(A^*_U) @>{j_*}>\simeq> K^*(A^*_X).
\end{CD}
\end{equation*}
\end{proposition}

As explained above, the previous two propositions give

\begin{theorem}\label{commutativity}
Let $M \overset{\iota}{\to} X$ be a closed embedding of manifolds with
corners. Then the diagram
\begin{equation}\label{diag4}
\begin{CD}
    \kc{M)} @>{\iota_*}>> \kc{X}\\ @AA{\ind_a^M}A
    @A{\ind_a^X}AA\\ K^*(A^*_M) @>\iota_{!}>> K^*(A^*_X) \\
\end{CD}
\end{equation}
is commutative.
\end{theorem}

\section{A topological index theorem}

Motivated by Theorem \ref{commutativity} and by the results of Section
\ref{sec.prop} (see Propositions \ref{prop.2} and \ref{prop.3}) we
introduce the following definition.

\begin{definition}\label{def.class}\
A {\em strong classifying manifold} $X_M$ of $M$ is a compact manifold with corners
$X_M$, together with a closed embedding $\iota: M \to X_M$ with the
following properties:
\begin{enumerate}[(i)]
\item\ each open face of $X_M$ is diffeomorphic to a Euclidean space,
\item\ $F \to F \cap M$ induces a bijection between the open faces of
  $X_M$ and ~$M$.
\end{enumerate}
\end{definition}

Note that if $M \subset X_M$ are as in the above definition, then each
face of $M$ is the transverse intersection of $M$ with a face of
$X_M$.

\begin{proposition}
Let $M$ be a \mec\ with embedded faces, and $\iota : M \hookrightarrow
X_M$ be a strong classifying space of $M$. Then the maps $\iota_*$ and
$\ind_a^X$ of Theorem \ref{commutativity} are isomorphisms. That is, a
strong classifying space for $M$ is a classifying space for $M$.
\end{proposition}

\begin{proof}
This was proved in Propositions \ref{prop.2} and \ref{prop.3}.
\end{proof}

Let $\iota : M \to X_M$ be a classifying space for $M$. The above
proposition then allows us to define (see diagram \ref{diag4})
\begin{equation*}
    \ind_t^M :=\iota_*^{-1} \circ \ind_a^X \circ \iota_{!} :
    K^*(A^*_M) \to \kc{M}.
\end{equation*}

If $M$ is a smooth compact manifold (so, in particular, $\pa M =
\emptyset$), then $C^*(M) = \maK$, the algebra of compact operators on
$L^2(M)$ and hence $K_0(C^*(M)) = \ZZ$. Any embedding $\iota : M
\hookrightarrow \RR^N$ will then be a classifying space for
$M$. Moreover, for $X = \RR^n$, the map $\iota_*^{-1} \circ \ind_a^X :
K^*(TX) \to \ZZ$ is the inverse of $j_! : K^0(pt) \to K^0(T\RR^N)$ and
hence $\ind_t^{\RR^N} = (j_!)^{-1} \iota_!$, which is the definition
of the topological index from \cite{AS1}. In view of this fact, we
shall also call the map $\ind_t^M$ {\em the topological index}
associated to $M$.  Theorem \ref{commutativity} then gives the
following result:

\begin{theorem}\label{thm.topological}\ 
Let $M$ be a manifold with corners and $A_M$ and $C^*(M)$ be the Lie
algebroid and the $C^*$-algebra associated to $M$. Then the principal
topological index map $\ind_t^M$ depends only on $M$, that is, it is
independent of the classifying space $X_M$, and we have
\begin{equation*}
    \ind_t^M = \ind_a^M : K^*(A^*_M) \to \kc{M}.
\end{equation*}
\end{theorem}

If $M$ is a smooth compact manifold ({\em no boundary}), this recovers
the Atiyah-Singer index theorem on the equality of the full
$C^*$-analytic and principal symbol topological index \cite{AS1}.

\section{$K$-theory of comparison algebras}

The isomorphism $K_*(C^*(M)) \simeq K^*(X_M)$ provides us with a way
of determining the groups $K_*(C^*(M))$. In particular, we have
completed the determination of the $K$-theory groups of the comparison
algebra $\fa_{-1}(M_0) = C^*(M)$, if the interior of $M$ is a endowed
with a metric making it a manifold with poly-cylindrical ends.

\begin{theorem} 
Let $M$ be a manifold with corners and embedded faces. We have
$K_*(C^*(M)) \simeq K^*(X_M)$. Moreover $K_j(C^*(M)) \otimes \QQ
\simeq \QQ^{p_j}$, where $p_j$ is the number of faces of $M$ of
dimension $\equiv j$ modulo $2$.
\end{theorem}

The last part of the above theorem is proved by showing that the
Atiyah-Hirzebruch spectral sequence of $X_M$ collapses at $E^2$. This
is part of a joint work with Etienne Fieux.

It is not difficult to construct a classifying manifold $X_M$ of $M$
\cite{MN}, that is, a manifold such that $\iota_* : K_0(C^*(M)) \to
K_0(C^*(X))$ and $\ind_a^X : K^0(A_X^*) \to K_0(C^*(X))$ are
isomorphisms. Let us assume $M$ is compact with embedded faces.  The
space $X_M$ is obtained from an embedding $X \to [0,\infty)^N$ for
  some large $N$, and then by removing suitable hyperplanes from the
  boundary of $[0,\infty)^N$ such that each face of $M$ is the
    transverse intersection of $M$ and of a face of $X_M$. See also
    \cite{Savin2, Savin1}.

\bibliographystyle{plain}

\bibliography{mn}

\end{document}